\documentclass[12pt]{article}

\usepackage{t1enc}
\usepackage{amsthm,amsmath,amssymb}
\usepackage{url}
\usepackage[colorlinks=true,citecolor=black,linkcolor=black,urlcolor=blue]{hyperref}

\newcommand{\doi}[1]{\href{http://dx.doi.org/#1}{\texttt{doi:#1}}}

\newcommand{\higpig}{\mathcal{H}}
\newcommand{\field}{\mathbb{F}}
\newcommand{\prospace}{\mathbb{P}}
\newcommand{\ident}{\mathrm{id}}
\newcommand{\opp}{\mathrm{opp}}

\DeclareMathOperator{\PG}{\mathsf{PG}}

\DeclareMathOperator{\grass}{\mathbb{G}}
\DeclareMathOperator{\rank}{\mathrm{rank}}
\DeclareMathOperator{\mult}{\mathrm{mult}}
\DeclareMathOperator{\karak}{\mathrm{char}}

\theoremstyle{plain}
\newtheorem{theorem}{Theorem}
\newtheorem{lemma}[theorem]{Lemma}
\newtheorem{corollary}[theorem]{Corollary}
\newtheorem{proposition}[theorem]{Proposition}

\theoremstyle{definition}
\newtheorem{definition}[theorem]{Definition}

\newtheorem*{not*}{Notation}

\theoremstyle{remark}
\newtheorem{remark}[theorem]{Remark}

\title{\bf Higgledy-piggledy subspaces and~uniform subspace designs}

\author{
{Sz}abol{cs}~L.~Fan{cs}ali\thanks{This research was partially supported by the ECOST Action IC1104 and OTKA Grant~K81310}\\
\small MTA--ELTE Geometric~and~Algebraic~Combinatorics Research~Group\\[-0.8ex]
\small ELTE, Institute~of~Mathematics, Department~of~Geometry\\[-0.8ex]
\small Budapest, Hungary\\
\small\tt nudniq@cs.elte.hu\\
\and\addtocounter{footnote}{5}
P{\'e}ter~{Sz}iklai\thanks{This research was partially supported by the Bolyai Grant and OTKA Grant~K81310}\\
\small MTA--ELTE Geometric~and~Algebraic~Combinatorics Research~Group\\[-0.8ex]
\small ELTE, Institute~of~Mathematics, Department~of~Computer~Science\\[-0.8ex]
\small Budapest, Hungary\\
\small\tt sziklai@cs.elte.hu
}

\date{
\small September 22, 2014
}

\begin{document}


\maketitle


\begin{abstract}
In this article, we investigate collections of `well-spread-out' projective (and linear) subspaces. Projective $k$-subspaces in $\mathsf{PG}(d,\mathbb{F})$ are in `higgledy-piggledy arrangement' if they meet each projective subspace of co-dimension $k$ in a generator set of points. We prove that the set $\mathcal{H}$ of higgledy-piggledy $k$-subspaces has to contain \emph{more than} $\min\left\{|\mathbb{F}|,\sum_{i=0}^k\lfloor\frac{d-k+i}{i+1}\rfloor\right\}$ elements. We also prove that $\mathcal{H}$ has to contain \emph{more than} $(k+1)\cdot(d-k)$ elements if the field $\mathbb{F}$ is algebraically closed.

An $r$-uniform \emph{weak} $(s,A)$ subspace design is a set of linear subspaces $H_1,\dots,H_N\le\mathbb{F}^m$ each of rank $r$ such that each linear subspace $W\le\mathbb{F}^m$ of rank $s$ meets at most $A$ among them. This subspace design is an $r$-uniform \emph{strong} $(s,A)$ subspace design if $\sum_{i=1}^N\mathrm{rank}(H_i\cap W)\le A$ for $\forall W\le\mathbb{F}^m$ of rank $s$. We prove that if $m=r+s$ then the dual ($\{H_1^\bot,\dots,H_N^\bot\}$) of an $r$-uniform weak (strong) subspace design of parameter $(s,A)$ is an $s$-uniform weak (strong) subspace design of parameter $(r,A)$. We show the connection between uniform weak subspace designs and higgledy-piggledy subspaces proving that $A\ge\min\left\{|\mathbb{F}|,\sum_{i=0}^{r-1}\lfloor\frac{s+i}{i+1}\rfloor\right\}$ for $r$-uniform weak or strong $(s,A)$ subspace designs in $\mathbb{F}^{r+s}$.

We show that the $r$-uniform strong $(s,r\cdot s+\binom{r}{2})$ subspace design constructed by Guruswami and Kopprty (based on multiplicity codes) has parameter $A=r\cdot s$ if we consider it as a \emph{weak} subspace design. We give some similar constructions of weak and strong subspace designs (and higgledy-piggledy subspaces) and prove that the lower bound $(k+1)\cdot(d-k)+1$ over algebraically closed field is tight.
\end{abstract}


\section{Introduction}

In our previous article~\cite{higpig}, we examined sets $\mathcal{G}$ of points such that each hyperplane $\Pi$ is \emph{spanned} by the intersection $\Pi\cap\mathcal{G}$. Examination this question had been inspired by {H{\'e}ger, Patk{\'o}s and Tak{\'a}{ts}}~\cite{HTP}, who hunt for a set $\mathcal{G}$ of points in the projective space $\PG(d,q)$ that `determines' all hyperplanes in the sense that the intersection $\Pi\cap\mathcal{G}$ is \emph{individual} for each hyperplane $\Pi$.

A similar question is to find a set $\mathcal{G}$ of points such that each subspace $\Pi$ of co-dimension $k$ is spanned by the intersection $\Pi\cap\mathcal{G}$. For the sake of conciseness, a projective subspace of dimension $k$ will be called a projective $k$-subspace and a subspace of co-dimension $k$ will be called a co-$k$-subspace from now on.

\begin{definition}[Multiple $k$-blocking set]
A set $\mathcal{B}$ of points in the projective space $\PG(d,\field)$ is a $t$-fold \emph{blocking set} with respect to co-$k$-subspaces (briefly, a $t$-fold $k$-blocking set), if each projective subspace $\Pi<\PG(d,\field)$ of co-dimension $k$ meets $\mathcal{B}$ in at least $t$ points. A $t$-fold one-blocking set (i.e.~with respect to hyperplanes) is briefly said to be a $t$-fold blocking set.
\end{definition}

If $k>1$ then the definition of the $t$-fold $k$-blocking set does not say anything more about the intersections with the co-$k$-subspaces. In higher dimensions, a natural \emph{specialization} of multiple $k$-blocking sets would be the following.

\begin{definition}[$k$-generator set]
A set $\mathcal{G}$ of points in the projective space $\PG(d,\field)$ is a \emph{generator set} with respect to co-$k$-subspaces (or briefly, a \emph{$k$-generator set}), if each subspace $\Pi\subset\PG(d,\field)$ of co-dimension $k$ meets $\mathcal{G}$ in a `generator system' of $\Pi$, that is, $\mathcal{G}\cap\Pi$ spans $\Pi$, in other words this intersection is not contained in any hyperplane of $\Pi$. (Hyperplanes of co-$k$-subspaces are subspaces in $\PG(d,\field)$ of co-dimension $k+1$.)
\end{definition}

If the field $\field$ is finite then a $k$-generator set of points is finite, and so we can ask the minimal cardinality of a $k$-generator set as a combinatorial question. Since finitely many points could generate only finitely many subspaces, a $k$-generator set of points must be infinite if the field $\field$ is not finite. But it could be the union of finitely many geometric objects. Such type of $k$-generator sets well be investigated in the followings. Thus, we have combinatorial questions over arbitrary fields.

\subsection{Higgledy-piggledy subspaces}

{H{\'e}ger, Patk{\'o}s and Tak{\'a}{ts}}~\cite{HTP} had the idea to search generator set with respect to hyperplanes as the union of some disjoint projective lines. The generalization of this idea is to search generator set with respect to co-$k$-subspaces as the union of some (possibly disjoint) projective $k$-spaces. Note that the union of $t$ disjoint projective $k$-spaces is always a $t$-fold $k$-blocking set.

\begin{definition}[Higgledy-piggledy $k$-subspaces]
A set $\higpig$ of projective $k$-subspaces is a \emph{generator set} with respect to co-$k$-subspaces (briefly, a \emph{$k$-generator set} of $k$-subspaces), if the set $\bigcup\higpig$ of all points of the subspaces contained by $\higpig$ is a generator set with respect to co-$k$-subspaces. The elements of a \emph{$k$-generator set} of $k$-subspaces is said to be in \emph{higgledy-piggledy arrangement} and a $k$-generator set of $k$-subspaces is said to be a \emph{set of higgledy-piggledy $k$-subspaces}.
\end{definition}

The terminology `higgledy-piggledy arrangement' is introduced by {H{\'e}ger, Patk{\'o}s and Tak{\'a}{ts}}~\cite{HTP} in the case of `higgledy-piggledy lines'.

At first, we try to give another equivalent definition to the `higgledy-piggledy' property of $k$-generator sets of $k$-subspaces. The following is not an equivalent but a sufficient condition. Although, in several cases it is also a necessary condition (if we seek minimal such sets), thus, it could effectively be considered as an almost-equivalent.

\begin{theorem}[Sufficient condition]\label{thm:suff}
If there is no subspace of co-dimension $k+1$ meeting each element of the set $\higpig$ of $k$-subspaces then $\higpig$ is a generator set with respect to co-$k$-subspaces.
\end{theorem}
\begin{proof}
Suppose that the set $\higpig$ of $k$-subspaces is \emph{not} a generator set with respect to co-$k$-subspaces. Then there exists at least one co-$k$-subspace $\Pi$ that meets $\bigcup\higpig$ in a set $\Pi\cap\left(\bigcup\higpig\right)$ of points which is contained in a hyperplane $W$ of $\Pi$. Since $\Pi$ is of co-dimension $k$ it meets every projective $k$-subspace, thus each element of $\higpig$ meets $\Pi$, but the point(s) of intersection has (have) to be contained in $W$. Thus the subspace $W$ (of co-dimension $k+1$) meets each element of $\higpig$.
\end{proof}

The theorem above is a sufficient but not necessary condition. But if this condition above does not hold, then the set $\higpig$ of $k$-subspaces could only be a $k$-generator set in a very special way.

\begin{proposition}\label{prop:sundiszno}
If the set $\higpig$ of $k$-subspaces is a generator set with respect to co-$k$-subspaces of $\PG(d,\field)$ and there exists a subspace $W$ of co-dimension $k+1$ that meets each element of $\higpig$ then $\higpig$ has to contain at least as many elements as many points there are in a projective line. (That is, $|\higpig|\ge q+1$ if the field $\field=\field_q$ and $\higpig$ is infinite if the field $\field$ is \emph{not} finite.)
\end{proposition}
\begin{proof}
The points of the factor geometry $\PG(d,\field)/W\cong\PG(k,\field)$ are the co-$k$-subspaces of $\PG(d,\field)$ containing $W$. Let $H_i\in\higpig$ a $k$-subspace and consider the projective subspace $H_i\vee W$ spanned by $H_i$ and $W$. Since $W$ meets $H_i$, $H_i\vee W$ could not be the whole projective space. By factorization with $W$, $H_i\vee W$ becomes a proper projective subspace or the emptyset. A point $P$ of the factor geometry $\PG(d,\field)/W\cong\PG(k,\field)$ as a co-$k$-subspace $\hat{P}$ of $\PG(d,\field)$ could only be generated by $\higpig$ only if there exists a $k$-subspace $H_i\in\higpig$ such that $\hat{P}\cap H_i$ is not contained in $W$, that is, $H_i\vee W$ as a subspace (of the factor geometry) contains $P$. Thus, if $\higpig$ is a generator set of $k$-subspaces with respect to co-$k$-subspaces, and $\higpig$ is blocked by the subspace $W$ of co-dimension $k+1$, then $\higpig$ is a set of proper subspaces of the factor geometry $\PG(d,\field)/W\cong\PG(k,\field)$, covering all points of this projective space.

Extending these proper subspaces to hyperplanes of the factor geometry, and then consider these hyperplanes as the points of the \emph{dual geometry} $\left(\PG(d,\field)/W\right)^*$, we have a blocking set with respect to hyperplanes of $\left(\PG(d,\field)/W\right)^*$, thus, it has to contain at least as many elements as many points there are in a projective line.
\end{proof}

\begin{corollary}\label{cor:ekv}
Suppose that the set $\higpig$ of $k$-subspaces has \emph{at most} $|\field|$ elements. Then $\higpig$ is a set of higgledy-piggledy $k$-subspaces if and only if there is no subspace of co-dimension $k+1$ meeting each element of $\higpig$.\qed
\end{corollary}

Thus, the sufficient condition in Theorem~\ref{thm:suff} is an
\emph{equivalent} condition if $|\higpig|\le|\field|$, that's why we
called it `almost-equivalent'.

\begin{remark}
If $\higpig$ is a set of (much more than $N$) projective $k$-subspaces such that there is no subspace $W$ of co-dimension $k+1$ meeting at least $N$ elements of $\higpig$ then arbitrary $N$ elements of $\higpig$ are in higgledy-piggledy arrangement.
\end{remark}


\subsection{Uniform weak subspace designs}

A similar (but not identical) property is called `well-spread-out' by {Guruswami and Kopparty} in~\cite{VGSK} where they gave the definition~\cite[Definition~2]{VGSK} of \emph{weak $(s,A)$ subspace designs}. Since we are interested in subspace designs containing subspaces of the same dimension, we define the \emph{uniform subspace designs}. From now on, we use the word \emph{rank} in linear context and the word \emph{dimension} in projective context exclusively, to avoid confusion.

\begin{definition}[Uniform weak subspace design]
A collection $\{H_1,\dots,H_N\}$ of linear subspaces of rank $r$ in the vector space $\field^m$ is called an \emph{$r$-uniform weak $(s,A)$ subspace design} if for every linear subspace $W\subset\field^{m}$ of rank $s$, the number of indices $i$ for which $\rank(H_i\cap W)>0$ is at most $A$.
\end{definition}

This definition is not meaningless only if the subspace design contains at least $N\ge A+1$ subspaces. Since a linear subspace $W$ of rank $s$ and a linear subspace $H$ of rank $r$ always meet each other nontrivially in the vector space $\field^m$ if $s+r>m$, the parameter $r$ should be at most $m-s$ if we seek nontrivial $r$-uniform $(s,A)$ subspace designs.

The standard scalar product $\langle a|b\rangle=\sum_{i=0}^{m-1}a_ib_i$ makes the isomorphism $(\field^m)^*\equiv\field^m$ canonical, in other words, the vector space $\field^m$ is self-dual. Let $H^{\bot}=\{a\in(\field^m)^*\equiv\field^m\;|\;\langle a|b\rangle=0:\forall b\in H\}$ denote the annihillator (orthogonal complementary) subspace of $H\le\field^m$ in $(\field^m)^*\equiv\field^m$. If $\rank H=r$ then $\rank H^{\bot}=m-r$.

If the parameter $r$ equals to $m-s$, then the dual of an $r$-uniform subspace design $\higpig$ (containing the annihillators of the elements of $\higpig$) is again a uniform subspace design.

\begin{theorem}\label{thm:weakdual}
If $\{H_1,\dots,H_N\}$ is an $(m-s)$-uniform weak $(s,A)$ subspace design in the linear space $\field^m$ of rank $m$ then the collection $\{H_1^{\bot},\dots,H_N^{\bot}\}$ of co-$(m-s)$-subspaces in the dual vector space $(\field^m)^*$ is an $s$-uniform weak $(m-s,A)$ subspace design.
\end{theorem}
\begin{proof}
The linear subspace $W$ of rank $s$ and a linear subspace $H$ of rank $m-s$ meet each other nontrivially in the vector space $\field^m$ if and only if there exists a hyperplane $\Pi$ contaning both $H$ and $W$. In the dual space $(\field^m)^*$ it means that the one-dimensional subspace $\Pi^{\bot}$ is contained by both $H^{\bot}$ and $W^{\bot}$.
\end{proof}

If the parameter $r$ is less than $m-s$ then the dual of an $r$-uniform $(s,A)$ subspace design in the linear space $\field^m$ is not necesseraly a nontrivial subspace design.

If the weak $(s,A)$ subspace design $\higpig=\{H_1,\dots,H_N\}$ is non-uniform (that is, for each linear subspace $W<\field^m$ of rank $s$ there exist at most $A$ elements of $\higpig$ meeting $W$ non-trivially but for example $\rank H_1\neq\rank H_2$) then its dual is not necesseraly a weak subspace design at all.

The following proposition makes connection between uniform weak subspace designs and higgledy-piggledy ($k$-generator) subspaces.

\begin{proposition}\label{prop:connection}
If the set $\{H_1,\dots,H_N\}$ of \emph{linear} subspaces is a $(k+1)$-uni\-form weak $(d-k,A)$ subspace design in the vector space $\field^{d+1}$ then arbitrary subset $\higpig$ of at least $A+1$ elements (among $H_1,\dots,H_N$) is a set of \emph{projective} $k$-subspaces in higgledy-piggledy arrangement.

And conversely, suppose that $\{H_1,\dots,H_N\}$ is a set of \emph{projective} $k$-sub\-spaces in $\PG(d,\field)$ and there exists a finite positive integer $A<|\field|$ such that for each subset $\higpig\subset\{H_1,\dots,H_N\}$: if $|\higpig|=A+1$ then $\higpig$ is a set of $k$-subspaces in higgledy-piggledy arrangement. In this case $\{H_1,\dots,H_N\}$ is a set of \emph{linear} subspaces constituting a $(k+1)$-uniform weak $(d-k,A)$ subspace design in the vector space $\field^{d+1}$.
\end{proposition}
\begin{proof}
If $\{H_1,\dots,H_N\}$ is a $(k+1)$-uniform weak $(d-k,A)$ subspace design in $\field^{d+1}$ and $\higpig$ is a subset of at least $A+1$ elements among $H_1,\dots,H_N$ then for each linear subspace $W<\field^{d+1}$ of rank $d-k$ (i.e. $W<\PG(d,\field)$ is of co-dimension $k+1$) there exists at least one element of $\higpig$ disjoint to $W$ in projective sense (or meeting $W$ trivially in linear sense). So, $\higpig$ satisfies the sufficient condition of Theorem~\ref{thm:suff}.

Suppose that each subset $\higpig\subset\{H_1,\dots,H_N\}$ of cardinality $A+1$ is a set of projective $k$-subspaces in higgledy-piggledy arrangement. Since $A$ is \emph{less} then the cardinality of the field $\field$, then $|\higpig|=A+1$ is less than the cardinality of a projective line, and thus, Proposition~\ref{prop:sundiszno} concludes that there cannot exist a projective subspace $W$ of co-dimension $k+1$ (i.e. a linear subspace of rank $d-k$) meeting each element of $\higpig$. Thus, for each linear subspace $W$ of rank $d-k$ there are at most $A$ elements of $\{H_1,\dots,H_N\}$ meeting $W$ nontrivially.
\end{proof}

\subsection{Uniform strong subspace designs}

{Guruswami and Kopparty} defined the \emph{strong $(s,A)$ subspace designs} in~\cite[Definition~3]{VGSK}. One can define the strong subspace designs containing same rank subspaces as follows.

\begin{definition}[Uniform strong subspace design]
A collection $\{H_1,\dots,H_M\}$ of linear subspaces of rank $r$ in the vector space $\field^m$ is called an \emph{$r$-uniform strong $(s,A)$ subspace design} if for every linear subspace $W\subset\field^m$ of rank $s$, the sum $\sum_{i=1}^M\rank(H_i\cap W)$ is at most $A$.
\end{definition}

\begin{remark}
As it mentioned also by {Guruswami and Kopparty}~\cite{VGSK}, every $r$-uniform strong $(s,A)$ subspace design is also an $r$-uniform weak $(s,A)$ subspace design, and every $r$-uniform weak $(s,A)$ subspace design is also an $r$-uniform strong $(s,\min\{sA,rA\})$ subspace design.
\end{remark}

{Guruswami and Kopparty}~\cite{VGSK} constructed $r$-uniform strong $(s,A)$ subspace designs in the vector space $\field_q^m$ over the finite field of $q>m$ elements. Their first construction~\cite[Section~4]{VGSK} is based on Reed--Solomon codes. Their second construction~\cite[Section~5]{VGSK} is based on multiplicity codes but this second construction works only if $\karak\field_q>m$.
Translating the notation of Guruswami's and Kopparty's work~\cite{VGSK} to the slightly different convention of this article, \cite[Theorem~14, Theorem~17 and Theorem~20]{VGSK} say that $A\le\frac{(m-1)s}{m-r-h(s-1)}$ for both constructions. The work~\cite{higpig} sharpened these results as follows.

\begin{theorem}[{\rm\cite[Theorem~38]{higpig}}]\label{thm:GKimp}
$A\le\frac{(m-\frac{s+1}{2})s}{m-r-h(s-1)}$ for the first Guruswami--Kopparty construction, and $A\le\frac{(m-s)s}{m-r-h(s-1)}$ for the second Guruswami--Kopparty construction.\qed
\end{theorem}

\begin{remark}
The parameter what we denote here $h$, comes from the trick applying the extension field $\field_{q^h}$ during the constructions. The basic case is $h=1$. The constraint $m-r>h(s-1)$ results the bounds $\frac{m-r}{s-1}>h>0$ if $s>1$.
\end{remark}

\begin{theorem}\label{thm:strongdual}
The dual of a $(m-s)$-uniform \emph{strong} $(s,A)$ subspace design in $\field^m$ is an $s$-uniform \emph{strong} $(m-s,A)$ subspace desing in $(\field^m)^*\equiv\field^m$.
\end{theorem}
\begin{proof}
Let the set $\{H_1,\dots,H_N\}$ of linear subspaces ($H_i<\field^m$, $\rank H_i=m-s$) be a uniform strong $(s,A)$ subspace desing, that is, for each linear subspace $W<\field^m$ of rank $s$, $\sum_{i=1}^{N}\rank(W\cap H_i)\le A$.

Then the linear subspaces $H^{\bot}_1,\dots,H^{\bot}_N$ in $(\field^m)^*\equiv\field^m$ are of rank $s$ and for each linear subspace $V<(\field^m)^*$ of rank $m-s$ there exists a linear subspace $W<\field^m$ of rank $s$ such that $V=W^{\bot}$.

We know that $\rank V^{\bot}=m-\rank V$ and $(W\vee H_i)^{\bot}=W^{\bot}\cap H_i^{\bot}$, thus, $m-\rank(W\vee H_i)=\rank(W\vee H_i)^{\bot}=\rank(W^{\bot}\cap H_i^{\bot})$.

Since $\rank(W\cap H_i)=\rank W +\rank H_i -\rank(W\vee H_i)=s+(m-s)-\rank(W\vee H_i)=m-\rank(W\vee H_i)=\rank(W^{\bot}\cap H_i^{\bot})$ then the sum $\sum_{i=1}^{N}\rank(W^{\bot}\cap H^{\bot}_i)=\sum_{i=1}^{N}\rank(W\cap H_i)\le A$ for each linear subspace $W^{\bot}<(\field^m)^*$ of rank $m-s$.
\end{proof}

In this article, we are interested in $r$-uniform strong or weak $(s,A)$ subspace designs in $\field^m$ where $r+s=m$. If $s>1$ (and $m=r+s$) then the bound $\frac{m-r}{s-1}>h>0$ has the form $1+\frac{1}{s-1}=\frac{s}{s-1}>h>0$, thus, in this case $h=1$ is required in the Guruswami--Kopparty constructions. Thus, the first Guruswami--Kopparty construction gives us an $r$-uniform strong $(s,r\cdot s+\binom{s}{2})$ subspace design; and the second Guruswami--Kopparty construction (working if $\karak\field_q>m=r+s$) gives us an $r$-uniform strong $(s,r\cdot s)$ subspace design.

\begin{corollary}
For given $s\ge2$ and $r\ge2$ there exist an $r$-uniform strong $(s,r\cdot s+\min\{\binom{s}{2},\binom{r}{2}\})$ subspace design in the vector space $\field^{r+s}$ if the field $\field$ has more than $r+s$ elements. Moreover, for given $s\ge2$ and $r\ge2$ there exist an $r$-uniform strong $(s,r\cdot s)$ subspace design in the vector space $\field^{r+s}$ if the characteristic $\karak\field$ of the field $\field$ is bigger than $r+s$.
\end{corollary}
\begin{proof}
Theorem~\ref{thm:strongdual} above says that the duals of the first and second Gu\-rus\-wa\-mi--Kopparty constructions are $s$-uniform strong $(r,r\cdot s+\binom{s}{2})$ and $s$-uniform strong $(r,r\cdot s)$ subspace designs, respectively. The second construction works only if $\karak\field_q>m=r+s$, but the first construction and its dual work over a field $\field$ of \emph{arbitrary characteristic} if $\field$ has more than $r+s$ elements.

The first Guruswami--Kopparty construction gives us an $r$-uniform strong $(s,r\cdot s+\binom{s}{2})$ subspace design and an $s$-uniform srong $(r,s\cdot r+\binom{r}{2})$ subspace design. The dual of this last design is an $r$-uniform strong $(s,s\cdot r+\binom{r}{2})$ subspace design.
\end{proof}

\subsection{Lower bound over arbitrary (large enough) fields}

In our previous work~\cite{higpig}, we proved the following lemma.

\begin{lemma}{\rm\cite[Lemma~13]{higpig}}\label{lem:lines}
If the set $\mathcal{L}$ of lines in $\PG(d,\field)$ has at most $\left\lfloor\frac{d}{2}\right\rfloor+d-1$ elements then there exists a subspace $H$ of co-dimension two meeting each line in $\mathcal{L}$.
\end{lemma}

This lemma can be generalized by induction as follows.

\begin{lemma}\label{lem:lower}
If the set $\higpig$ of $k$-subspaces in $\PG(d,\field)$ has at most $\left\lfloor\frac{d}{k+1}\right\rfloor+\left\lfloor\frac{d-1}{k}\right\rfloor+\dots+\left\lfloor\frac{d-k+1}{2}\right\rfloor+d-k$ elements then there exists a subspace $W$ of co-dimension $k+1$ meeting each subspace in $\higpig$.
\end{lemma}
\begin{proof}
Suppose by induction that for each $m$, at most $\left\lfloor\frac{m}{k}\right\rfloor+\left\lfloor\frac{m-1}{k-1}\right\rfloor+\dots+\left\lfloor\frac{m-k+2}{2}\right\rfloor+m-(k-1)$ subspaces of dimension $k-1$ in $\PG(m,\field)$ always be blocked by a subspace $W$ of co-dimension $k$. Lemma~\ref{lem:lines} says that this base of induction holds for $k=2$.

Let $H_1,\dots,H_{\lfloor\frac{d}{k+1}\rfloor}$ and $H_{\lfloor\frac{d}{k+1}\rfloor+i}$ ($1\le i\le\left\lfloor\frac{d-1}{k}\right\rfloor+\dots+\left\lfloor\frac{d-k+1}{2}\right\rfloor+d-k$) denote the elements of $\higpig$. There exists a subspace of dimension at most\linebreak $(k+1)\left\lfloor\frac{d}{k+1}\right\rfloor-1$ containing the planes $H_1,\dots,H_{\lfloor\frac{d}{k+1}\rfloor}$ so there exists a hyperplane $\Pi$ containing them. The hyperplane $\Pi$ meets each $k$-subspace in a subspace of dimension at least $k-1$, thus let $L_i\le\Pi\cap H_{i+\lfloor\frac{d}{k+1}\rfloor}$ be a $(k-1)$-subspace for $i=1,\dots,\left\lfloor\frac{m}{k}\right\rfloor+\left\lfloor\frac{m-1}{k-1}\right\rfloor+\dots+\left\lfloor\frac{m-k+2}{2}\right\rfloor+m-(k-1)$, where $m=d-1$.

By induction there exists a subspace $W<\Pi$ of co-dimension $k$ (co-dimension with respect to $\Pi$), that meets each subspace $L_i$ above, so $W$ meets the subspaces $H_{\lfloor\frac{d}{k+1}\rfloor+i}$, $1\le i\le\left\lfloor\frac{d-1}{k}\right\rfloor+\left\lfloor\frac{d-2}{k-1}\right\rfloor+\dots+\left\lfloor\frac{d-k+1}{2}\right\rfloor+d-k$. Subspaces $H_1,\dots,H_{\lfloor\frac{d}{k+1}\rfloor}$ are contained in $\Pi$, and $W$ has co-dimension $k$ in $\Pi$, thus, $W$ meets them also. The subspace $W$ has co-dimension $k+1$ in $\PG(d,\field)$ and it meets all the elements of $\higpig$.
\end{proof}

\begin{theorem}[Lower bound]\label{thm:lowerbound}
A generator set $\higpig$ of $k$-subspaces in $\PG(d,\field)$ has to contain at least $\min\left\{|\field|,\sum_{i=0}^{k}\left\lfloor\frac{d-k+i}{i+1}\right\rfloor\right\}+1$ elements.
\end{theorem}
\begin{proof}
If there exists a projective subspace $W$ of co-dimension $k+1$ meeting each element of $\higpig$ then Proposition~\ref{prop:sundiszno} says that $|\higpig|>|\field|$.

If there \emph{does not} exist any projective subspace $W$ of co-dimension $k+1$ meeting each element of $\higpig$ then Lemma~\ref{lem:lower} gives the result.
\end{proof}

As a consequence of this lower bound we get a bound for the parameter $A$ of weak $r$-uniform $(s,A)$ subspace designs.

\begin{corollary}
If the field $\field$ has at least $\left\lfloor\frac{m-1}{r}\right\rfloor+\left\lfloor\frac{m-2}{r-1}\right\rfloor+\dots+\left\lfloor\frac{m-r+1}{2}\right\rfloor+m-r+1$ elements, then for each $r$-uniform weak $(m-r,A)$ subspace design in $\field^m$, the parameter $A$ has to be at least $\left\lfloor\frac{m-1}{r}\right\rfloor+\left\lfloor\frac{m-2}{r-1}\right\rfloor+\dots+\left\lfloor\frac{m-r+1}{2}\right\rfloor+m-r$.
\end{corollary}
\begin{proof}
Let $d=m-1$ and $k=r-1$. Proposition~\ref{prop:connection} says that arbitrary $A+1$ elements of a $(k+1)$-uniform weak $(d-k,A)$ subspace design (in $\field^{d+1}$) are projective $k$-subspaces (of $\PG(d,\field)$) in higgledy-piggledy position. Theorem~\ref{thm:lowerbound} concludes that $A+1\ge\left\lfloor\frac{d}{k+1}\right\rfloor+\left\lfloor\frac{d-1}{k}\right\rfloor+\dots+\left\lfloor\frac{d-k+1}{2}\right\rfloor+d-k+1$.
\end{proof}


\section{Gra{ss}mann--Pl{\"u}cker coordinates}

Let $\grass(r,s,\field)$ or simply $\grass(r,s)$ denote the Gra{ss}mannian of the linear subspaces of rank $r$ (and so, of co-dimension $s$) in the vector space $\field^{r+s}$, or, in other aspect $\grass(r,s)$ is the set of all projective subspaces of dimension $r-1$ (and co-dimension $s$) in $\PG(r+s-1,\field)$.

\paragraph{Pl{\"u}cker embedding}
Let $H<\field^{r+s}$ be a linear subspace of rank $r$ and let $a(1),\dots,a(r)$ and $b(1),\dots,b(r)$ be two arbitrary bases of $H$. Let $L<\field^{r+s}$ be another linear subspace of rank $r$ ($H\neq L$) and let $c(1),\dots,c(r)$ be a basis of $L$. Since $a(1)\wedge\dots\wedge a(r)=\lambda\cdot b(1)\wedge\dots\wedge b(r)\neq0$ for a suitable nonzero $\lambda\in\field$, and since $c(1)\wedge\dots\wedge c(r)\neq0$ is \emph{not} the element of the subspace of rank one, generated by $a(1)\wedge\dots\wedge a(r)$, this subspace $\{a(1)\wedge\dots\wedge a(r)\cdot\lambda\;|\;\lambda\in\field\}<\bigwedge^r\field^{r+s}$ of rank one can be identified with $H$. This `Pl{\"u}cker embedding' identifies the Gra{ss}mannian with the set of rank-one linear subspaces of $\bigwedge^r\field^{r+s}$ generated by totally decomposable multivectors, that is, $\grass(r,s)\subset\PG\left(\bigwedge^r\field^{r+s}\right)\equiv\PG(\binom{r+s}{r}-1,\field)$ is an algebraic variety of dimension $r\cdot s$.

\paragraph{Pl{\"u}cker coordinates}
Let $\mathbf{H}\in\bigwedge^r\field^{r+s}$ denote a homogeneous coordinate vector of a point in $\PG(\binom{r+s}{r}-1,\field)$. If $\mathbf{H}\in\bigwedge^r\field^{r+s}$ is nonzero and totally decomposable (i.e.~$\mathbf{H}=a(1)\wedge\dots\wedge a(r)\neq0$ is a multivector for suitable vectors $a(i)\in\field^{r+s}$) then $\mathbf{H}$ is called the `Pl{\"u}cker coordinate vector' of the subspace $H$ generated by the vectors $a(i)\in\field^{r+s}$. The coordinates $H_{i_1\dots i_r}\in\field$ ($0\le i_1\dots i_r\le r+s-1$) of the multivector $\mathbf{H}$ are called the Pl{\"u}cker coordinates of the subspace $H<\field^{r+s}$ of rank $r$.

\paragraph{Pl{\"u}cker relations}
The numbers $H_{i_1\dots i_r}\in\field$ ($0\le i_1\dots i_r\le r+s-1$) are the coordinates of a totally decomposable multivector $\mathbf{H}$ (i.e.~the Pl{\"u}cker coordinates of the subspace $H<\field^{r+s}$) if and only if for each $2r$-tuple $(i_1,\dots,i_{r-1},j_0,j_1,\dots,j_r)$ of indices (each of them between zero and $r+s-1$)
\begin{equation}\label{eq:pl}
\sum_{n=0}^{r}(-1)^nH_{i_1\dots i_{r-1}j_n}H_{j_0\dots\hat{j_n}\dots j_r}=0\tag{P1}
\end{equation}
where the notation $j_0\dots\hat{j_n}\dots j_r$ means that the symbol $j_n$ is missing from the list $j_0\dots j_r$ of symbols. These quadratic equations are called `Pl{\"u}cker relations' and according to~\cite[Theorem~3.1.6.]{SCH}, the Pl{\"u}cker relations completely determine the Gra{ss}mannian $\grass(r,s)\subset\PG(\binom{r+s}{r}-1,\field)$, moreover, they generate the ideal of polynomials vanishing on it.

The following property of the Pl{\"u}cker relations will play a key role later.
\begin{lemma}\label{lem:szummas}
Consider the Pl{\"u}cker coordinates of the elements of $\grass(r,s)$ or $\grass(s,r)$, and let $N$ be a positive integer between $r$ and $r\cdot s$. Let the integers $i_1,\dots,i_r,j_1,\dots,j_r$ be given such that $i_1+\dots+i_r=N=j_1+\dots+j_r$ and $0\le i_1<\dots<i_r\le r+s-1$ and $0\le j_1<\dots<j_r\le r+s-1$ and suppose that $(i_1,\dots,i_r)\neq(j_1,\dots,j_r)$. Then there exists a Pl{\"u}cker relation that has the form
\begin{equation*}
H_{i_1\dots i_r}H_{j_1\dots j_r}+\Sigma=0
\end{equation*}
where $\Sigma$ is the sum of some products $H_{k_1\dots k_r}H_{n_1\dots n_r}$, where $k_1+\dots+k_r<N<n_1+\dots+n_r$.
\end{lemma}
\begin{proof}
Since $(i_1,\dots,i_r)\neq(j_1,\dots,j_r)$, there exists an index $i_\ell\notin\{j_1,\dots,j_r\}$. There exists an \emph{even} permutation $\sigma\in S_r$ such that $\sigma r=\ell$. The Pl{\"u}cker relation according to the $2r$-tuple $(i_{\sigma1}\dots i_{\sigma(r-1)},j_0=i_{\sigma r},j_1,\dots,j_r)$ is
\begin{equation}\label{eq:mp}
H_{i_{\sigma1}\dots i_{\sigma r}}H_{j_1\dots j_r}+
\sum_{n=1}^{r}(-1)^nH_{i_{\sigma1}\dots i_{\sigma(r-1)}j_n}H_{j_0\dots\hat{j_n}\dots j_r}=0\tag{P2}
\end{equation}
using the notation $j_0=i_{\sigma r}=i_\ell$ and separating the first term of the sum.

Because $j_0=i_\ell,j_1,\dots,j_r$ are $r+1$ distinct elements, $j_n$ is either strictly less or strictly greater than $j_0$ if $n\neq0$, and thus, $(\sum_{k=0}^rj_k)-j_n$ is either strictly greater or strictly less than $N=\sum_{k=1}^rj_k$ if $n\neq0$. And, since $i_{\sigma1}+\dots+i_{\sigma(r-1)}+i_{\sigma r}+j_1+\dots+j_r=2N$,
if $(\sum_{k=0}^rj_k)-j_n>N$ then $i_{\sigma1}+\dots+i_{\sigma(r-1)}+j_n<N$, and conversely, if $(\sum_{k=0}^rj_k)-j_n<N$ then $i_{\sigma1}+\dots+i_{\sigma(r-1)}+j_n>N$.

And last, since $\sigma$ is even, then $H_{i_1\dots i_r}=H_{i_{\sigma1}\dots i_{\sigma r}}$, we get that the Pl{\"u}cker relation~\ref{eq:mp} is in the required form.
\end{proof}

\begin{remark}
If $\{i_1,\dots,i_r\}=\{j_1,\dots,j_r\}$ then each Pl{\"u}cker relation containing the product $H_{i_1\dots i_r}H_{j_1\dots j_r}$ reduces to $0=0$.
\end{remark}

\subsection{Dual Pl{\"u}cker coordinates}

Since the map $\star:\grass(s,r)\rightarrow\grass(r,s):W\mapsto W^{\bot}$ is a bijection between the linear subspaces of rank $r$ and the linear subspaces of rank $s$, we could define `dual' Pl{\"u}cker coordinates of rank-$s$ subspaces having $r$ indices instead of $s$.

Let $W<\field^{r+s}$ be an arbitrary subspace of rank $s$ and let $W^{\bot}<\field^{r+s}$ its orthogonal complementary subspace of rank $r$. The Pl{\"u}cker coordinate vector $\mathbf{W^\star}$ of the orthogonal complementary subspace $W^{\bot}$ is called the `dual Pl{\"u}cker coordinate vector' of the subspace $W$. The coordinates $W^\star_{i_1\dots i_r}\in\field$ ($0\le i_1\dots i_r\le r+s-1$) of the multivector $\mathbf{W^\star}$ are called the `dual Pl{\"u}cker coordinates' of the subspace $W<\field^{r+s}$ of rank $s$.

Although the Pl{\"u}cker coordinate vector of the subspace $W$ is denoted by $\mathbf{W}$, the Pl{\"u}cker coordinate vector of its orthogonal complementary subspace $W^\bot$ is denoted by $\mathbf{W^\star}$ instead of $\mathbf{W^\bot}$ because we want to avoid confusion between similar notations and the notation $\{\mathbf{W}\}^\bot$ means the hyperplane of the vector space $\left(\bigwedge^s\field^{r+s}\right)^*$ orthogonal to the vector $\mathbf{W}\in\bigwedge^s\field^{r+s}$.


The standard scalar product of the outer power space $\bigwedge^r\field^{r+s}$ is defined by the following identity
\begin{equation*}
\langle a(1)\wedge\dots\wedge a(r)|b(1)\wedge\dots\wedge b(r)\rangle=
\begin{vmatrix}
\langle a(1)|b(1)\rangle&\dots&\langle a(1)|b(r)\rangle\\
\vdots&\ddots&\vdots\\
\langle a(1)|b(r)\rangle&\dots&\langle a(r)|b(r)\rangle
\end{vmatrix}
\end{equation*}
which is defined only for multivectors but it can be extended consistently. This makes the first isomorphism in $(\bigwedge^r\field^{r+s})^*\equiv\bigwedge^r(\field^{r+s})^*\equiv\bigwedge^r\field^{r+s}$ canonical. The second canonical isomorphism comes from the self-duality $(\field^{r+s})^*\equiv\field^{r+s}$.

\begin{lemma}
Let $W<\field^{r+s}$ be an arbitrary subspace of rank $s$ and let $\mathbf{W^\star}$ denote its \emph{dual} Pl{\"u}cker coordinate vector. Let $H<\field^{r+s}$ be an arbitrary subspace of rank $r$ and let $\mathbf{H}$ denote its Pl{\"u}cker coordinate vector. Then $H\cap W\neq\{0\}\Leftrightarrow\langle\mathbf{W^\star}|\mathbf{H}\rangle=0$.
\end{lemma}
\begin{proof}
Let $a(1),\dots,a(s)$ be a basis of $W$ such that $\mathbf{W}=a(1)\wedge\dots\wedge a(s)$. Let $a(s\!+\!1),\dots,a(s\!+\!r)$ be a basis of $W^{\bot}$ such that $\mathbf{W^\star}=a(s\!+\!1)\wedge\dots\wedge a(s\!+\!r)$. And finally, let $b(1),\dots,b(r)$ be a basis of $H$ such that $\mathbf{H}=b(1)\wedge\dots\wedge b(s)$. $H\cap W\neq\{0\}$ if and only if $\exists v\in H$ such that $v\neq0$ and $v\bot W^{\bot}$. That is, $H\cap W\neq\{0\}$ if and only if $\alpha_1b(1)+\dots+\alpha_r b(r)\bot a(s+i)$ for each $i=1,\dots,r$. This is equivalent with the equation
\begin{equation*}
\begin{vmatrix}
\langle a(s+1)|b(1)\rangle&\dots&\langle a(s+1)|b(r)\rangle\\
\vdots&\ddots&\vdots\\
\langle a(s+1)|b(r)\rangle&\dots&\langle a(s+r)|b(r)\rangle
\end{vmatrix}=0
\end{equation*}
Thus, $H\cap W\neq\{0\}$ if and only if $\langle\mathbf{W^\star}|\mathbf{H}\rangle=0$.
\end{proof}

Since the standard basis of the outer power space $\bigwedge^r\field^{r+s}$ is $\{e(i_1)\wedge\dots\wedge e(i_r)\;|\;0\le i_1<\dots<i_r<r+s\}$, where $\{e(0),\dots,e(r\!+\!s\!-\!1)\}$ is the standard basis of $\field^{r+s}$, the standard scalar product is
\begin{equation*}
\langle\mathbf{L}|\mathbf{H}\rangle=\sum_{0\le i_1<\dots<i_r<r+s}L_{i_1,\dots,i_r}H_{i_1,\dots,i_r}.
\end{equation*}

\begin{lemma}
Let $\{\mathbf{H}(1),\dots,\mathbf{H}(N)\}$ denote the set of the Gra{ss}mann--Pl{\"u}cker \mbox{co-or}\-di\-nate vectors representing the elements of the set $\higpig$ of projective $(r-1)$-subspaces in $\PG(r+s-1,\field)$. There exists a subspace $W$ of co-dimension $r$ (rank $s$) in $\PG(r+s-1,\field)$ meeting each element of $\higpig$ \emph{if and only if} the subspace $\{\mathbf{H}(1)\}^{\bot}\cap\dots\cap\{\mathbf{H}(N)\}^{\bot}\le\PG(\binom{d+1}{r}-1,\field)$ meets the Gra{ss}mann variety $\grass(s,r)$, that is, the linear equation system
\begin{align*}
\sum_{i_1<\dots<i_r}H_{i_1\dots i_r}(1)W^\star_{i_1\dots i_r}&=0&\dots&&\sum_{i_1<\dots<i_r}H_{i_1\dots i_r}(N)W^\star_{i_1\dots i_r}&=0
\end{align*}
together with all the quadratic Pl{\"u}cker relations
\begin{equation*}
\sum_{n=0}^{r}(-1)^n W^\star_{i_1\dots i_{r-1}j_n}W^\star_{j_0\dots\hat{j_n}\dots j_r}=0
\end{equation*}
(for each $2r$-tuple $(i_1,\dots,i_{r-1},j_0,j_1,\dots,j_r)$ of indices) has nontrivial common solutions for the unknowns $W^\star_{i_1\dots i_r}$.
\end{lemma}
\begin{proof}
The quadratic Pl{\"u}cker relations completely determine the Pl{\"u}cker coordinates of the Gra{ss}mannian $\grass(r,s)\subset\PG(\binom{r+s}{r}-1,\field)$, and thus, they determine the dual Pl{\"u}cker coordinates of the Gra{ss}mannian $\grass(s,r)$.

So, if there exists a dual Pl{\"u}cker coordinate vector $\mathbf{W^\star}$ such that the scalar product $\langle\mathbf{W^\star}|\mathbf{H}(i)\rangle=0$ for each $i=1,\dots,r$ then the orthogonal complementary subspace $W$ of the subspace $W^{\bot}$ co-ordinatized by $\mathbf{W^\star}$ meets each subspace $H(i)$ co-ordinatized by $\mathbf{H}(i)$ nontrivially, and thus, $W$ meets each element of $\higpig$ nontrivialy.

Conversely, if there exists a subspace $W$ meeting each element of $\higpig$, then its dual Pl{\"u}cker coordinate vector $\mathbf{W^\star}$ is orthogonal to each $\mathbf{H}(i)$ and its coordinates satisfies the quadratic Pl{\"u}cker relations.
\end{proof}

\subsection{Lower bound over algebraically closed fields}

Since an algebraically closed field contains infinitely many elements, Corollary~\ref{cor:ekv} concludes that the finite set $\higpig$ of $k$-subspaces in $\PG(d,\field)$ over an algebraically closed field $\field$ could be a $k$-generator set if and only if the condition of Theorem~\ref{thm:suff} holds.

\begin{lemma}
{\rm\cite[Corollary~3.2.14 and Subsection~3.1.1]{SCH}}
The dimension of the Gra{ss}mannian as an algebraic variety is $\dim\grass(r,s)=r\cdot s$ and its degree is
\begin{equation*}
\deg\grass(r,s)=\frac{0!1!\dots(s\!-\!1)!}{r!(r\!+\!1)!\dots(r\!+\!s\!-\!1)!}\big(rs\big)!\tag*{\qed}
\end{equation*}
\end{lemma}

Remember that a projective algebraic variety $\grass\subset\prospace$ of \emph{dimension $n$} and a projective subspace $S\le\prospace$ of \mbox{\emph{co-dimension $n$}} always meet over an algebraically closed field.

\begin{theorem}
Over algebraically closed field $\field$, if the set $\higpig$ of $(r-1)$-subspaces in $\PG(r+s-1,\field)$ has at most $r\cdot s$ elements, then there exists a subspace $W$ in $\PG(r+s-1,\field)$ of co-dimension $s$ that meets each element of $\higpig$, and thus, $\higpig$ is \emph{not} an $(r-1)$-generator set.
\end{theorem}
\begin{proof}
Suppose that $\higpig=\{H_1,\dots,H_N\}$ has $N\le r\cdot s$ elements. If the subspace $H_i$ of rank $r$ is co-ordinatized by the homogeneous Pl{\"u}cker coordinate vector $\mathbf{H}(i)\in\PG(\binom{r+s}{r}-1,\field)$, then the Pl{\"u}cker coordinate vectors of co-$s$-subspaces meeting $H_i$ are the elements of the hyperplane $\{\mathbf{H}(i)\}^\bot<\PG(\binom{r+s}{r}-1,\field)$ orthogonal to the vector $\mathbf{H}(i)$.

The subspace
$\{\mathbf{H}(1)\}^{\bot}\cap\dots\cap\{\mathbf{H(N)}\}^{\bot}$ has co-di\-men\-sion at most $N\le r\cdot s$ in $\PG(\binom{r+s}{r}-1,\field)$ since it is the intersection of $N\le r\cdot s$ hyperplanes. The Gra{ss}mannian $\grass(r,s)$ of the $s$-co-dimensional subspaces of $\PG(r+s-1,\field)$ has dimension $r\cdot s$ and its degree $\deg\grass(r,s)$ is positive.

Thus, $\{\mathbf{H}(1)\}^{\bot}\cap\dots\cap\{\mathbf{H(N)}\}^{\bot}\cap\grass(s,r)$ contains at least $\deg\grass(r,s)\ge1$ elements, which are subspaces of co-dimension $r$ meeting all the subspaces in $\higpig$.
\end{proof}

So, using projective parameters $d=r+s-1$, $k=r-1$, $r\cdot s=(k+1)\cdot(d-k)$, we have shown that over algebraically closed field $\field$, the set $\higpig$ of higgledy-piggledy $k$-subspaces in $\PG(d,\field)$ has to contain at least $(k+1)\cdot(d-k)+1$ elements.


\section{Constructions based on the moment curve}

Let $\{(1,t,t^2,\dots,t^d):t\in\field\}\cup\{(0,0,0,\dots,1)\}\subset\PG(d,\field)$ be the moment curve (rational normal curve) and let $a(t)=(1,t,t^2,t^3\dots,t^d)$ denote the coordinate vectors of the points of the moment curve. In this section we investigate collections $\{H(t)\;|\;t\in\field\}$ of linear subspaces of rank $r$ (projective subspaces of dimension $k=r-1$) in the vector space $\field^{r+s}$ (in $\PG(d,\field)$ where $d=r+s-1$).

At first, we give a general description for the particular constructions given later in Subsection~\ref{subs:tang}, in Subsection~\ref{subs:div} and in Subsection~\ref{subs:sec}.

The subspace $H(t)$ is co-ordinatized by the Pl{\"u}cker coordinate vector $\mathbf{H}(t)=a^{[0]}(t)\wedge a^{[1]}(t)\wedge\dots\wedge a^{[k]}(t)$, where the $i$-th coordinate of the vector $a^{[n]}(t)$ is $a^{[n]}_i(t)=h(i,n)\cdot t^{i-n}$, where $h(i,n)\in\field$ is independent from $t$ and $\forall i:h(i,0)=1$, thus $a^{[0]}(t)=a(t)$. (If $h(i,n)\neq0$ for some $i<n$, then the case $t=0$ shall be handled separately.)

The Pl{\"u}cker coordinates of $H(t)$ are the $r\times r$ subdeterminants of the following matrix.
\begin{equation*}
\begin{bmatrix}
1&t&t^2&\dots&t^{r-1}&\dots&t^d\\
h(0,1)\frac{1}{t}&h(1,1)&h(2,1)t&\dots&h(r-1,1)t^{r-2}&\dots&h(d,1)t^{d-1}\\
h(0,2)\frac{1}{t^2}&h(1,2)\frac{1}{t}&h(2,2)&\dots&h(r-1,2)t^{r-3}&\dots&h(d,2)t^{d-2}\vphantom{\Big|}\\
\vdots&\vdots&\vdots&\ddots&\vdots&&\vdots\\
h(0,k)\frac{1}{t^k}&h(1,k)\frac{t}{t^{k}}&h(2,k)\frac{t^2}{t^{k}}&\dots&h(k,k)&\dots&h(d,k)t^{d-r+1}\\
\end{bmatrix}
\end{equation*}
Choosing the $i_1$-th,\dots,$i_r$-th coloumns, $\frac{t^{i_1}t^{i_2}\cdots t^{i_r}}{t^{0}t^{1}\cdots t^k}$ can be
separated, and we get
\begin{align*}
H_{i_1,\dots,i_r}(t)&=\frac{t^{i_1}\cdot t^{i_2}\cdot\cdots\cdot t^{i_r}}{t^{0}\cdot t^{1}\cdot\cdots\cdot t^k}\cdot\det\begin{bmatrix}
1&1&\dots&1\\
h(i_1,1)&h(i_2,1)&\dots&h(i_r,1)\\
\vdots&\vdots&\ddots&\vdots\\
h(i_1,k)&h(i_2,k)&\dots&h(i_r,k)\\
\end{bmatrix}=\\
&=t^{i_1+i_2+\dots+i_r-\binom{r}{2}}\cdot h(i_1,i_2,\dots,i_r)
\end{align*}
using that $0+1+2+\dots+k=\frac{(k+1)k}{2}=\frac{r(r-1)}{2}=\binom{r}{2}$ and using the abbreviation $h(i_1,i_2,\dots,i_r)$.

Since $h(i_1,i_2,\dots,i_r)\cdot t^{i_1+i_2+\dots+i_r-\binom{r}{2}}$ is not meaningless for $t=0$ (even if $h(i,n)\neq0$ for some $i<n$), we should only see that the vector co-ordinatized by the coordinates $H_{i_1,\dots,i_r}(0)=h(i_1,i_2,\dots,i_r)\cdot 0^{i_1+i_2+\dots+i_r-\binom{r}{2}}$ is a totally decomposable multivector. One can see that $H_{i_1,\dots,i_r}(0)=0$ if $\{i_1,\dots,i_r\}\neq\{0,1,\dots,r-1\}$, and $H_{0,1,\dots,r-1}(0)=h(0,1,\dots,r-1)$. If $h(0,1,\dots,r\!-\!1)$ is nonzero then $H(0)=e_0\wedge\dots\wedge e_{r-1}$, where $e_i$ is the standard basis vector $[0,\dots,1,\dots,0]$. Thus, if $h(i,n)\neq0$ for some $i<n$ then we extend the set $\{H(t)\;|\;t\in\field\setminus\{0\}\}$ by the element $H(0)=e_0\wedge\dots\wedge e_{r-1}$.

So, we will deal with sets $\{H(t):t\in\field\}$ of subspaces where each subspace $H(t)$ is co-ordinatized by the Pl{\"u}cker coordinates $H_{i_1,\dots,i_r}(t)=h(i_1,\dots,i_r)\cdot t^{i_1+\dots+i_r-\binom{r}{2}}$, where $h(i_1,\dots,i_r)\in\field$ is independent from $t$. At first, we have a lemma about such sets.

\begin{lemma}\label{lem:ind}
Suppose that $\forall t\in\field:H_{i_1,\dots,i_r}(t)=h(i_1,\dots,i_r)\cdot t^{i_1+\dots+i_r-\binom{r}{2}}$ are the Pl{\"u}cker coordinates of the subspace $H(t)<\field^{r+s}$ and suppose that $|\field|>r\cdot s$. There does \emph{not} exist any subspace $W<\field^{r+s}$ of co-dimension $r$ (of rank $s$) meeting each $H(t)$ non-trivially \emph{if and only if} $h(i_1,\dots,i_r)\neq0$ for each $r$-tuple $i_1,\dots,i_r$.
\end{lemma}
\begin{proof}
At first, suppose that $h(i_1,\dots,i_r)\neq0$ for each $r$-tuple $i_1,\dots,i_r$ and suppose to the contrary that there exists a subspace $W$ of {co-di}\-men\-sion $r$ meeting each subspace $H(t)$. Let $W^\star_{i_1,\dots,i_r}$ ($0\le i_1<\dots<i_r\le d$) denote the \emph{dual} Pl{\"u}cker coordinates of $W$. For these Pl{\"u}cker coordinates we have Pl{\"u}cker relations
\begin{equation*}
\sum_{n=0}^{r}(-1)^nW^\star_{i_1\dots i_{r-1}j_n}\cdot W^\star_{j_0\dots\hat{j_n}\dots j_r}=0
\end{equation*}
for each $2r$-tuple $i_1,\dots,i_{r-1},j_0\dots j_r$ of indices.

The indirect assumpion means that
\begin{align*}
\sum_{i_1<\dots<i_r}W^\star_{i_1,\dots,i_r}\cdot H_{i_1,\dots,i_r}(t)=
\!\sum_{N=\binom{r}{2}}^{r\cdot d-\binom{r}{2}}t^{N-\binom{r}{2}}
\!\!\sum_{\substack{i_1+\dots+i_r=N\\i_1<\dots<i_r}}
\!\!h(i_1,\dots,i_r)\cdot W^\star_{i_1,\dots,i_r}
=0
\end{align*}
for all $t\in\field$.
Since the field $\field$ has more than $r\cdot s$ elements, this polynomial above can vanish on each element of $\field$ if only if its each coefficient is zero. So we have $r\cdot s+1$ new (linear) equations for the dual Pl{\"u}cker coordinates of $W$:
\begin{align*}
h(0,1,\dots,r-1)\cdot W^\star_{0,1,\dots,r-1}&=0\mbox{\tag{$N=\binom{r}{2}$}}\\
&\;\;\vdots\\
\vphantom{\sum_{N=\binom{r}{2}}^{r\cdot d-\binom{r}{2}}t^{N-\binom{r}{2}}}
\sum_{\substack{i_1+\dots+i_r=N\\ i_1<\dots<i_r}}h(i_1,\dots,i_r)\cdot W^\star_{i_1,\dots,i_r}&=0\mbox{\tag{$N$}}\\
&\;\;\vdots\\
h(d-r+1,\dots,d)\cdot W^\star_{d-r+1,\dots,d}&=0\mbox{\tag{$N=r\cdot d-\binom{r+1}{2}$}}\\
\end{align*}
Notice that in equation $(N)$ the sum of indices of each dual Pl{\"u}cker coordinate equals to $N$.

Suppose by induction that for each $r$-tuple $i_1,\dots,i_r$ if $i_1+\dots+i_r\le K$ then $W^\star_{i_1,\dots,i_r}=0$. The first equation then says that $W^\star_{0,1,\dots,r-1}=0$ so the base of induction holds for $K=\binom{r}{2}$.

Consider the dual Pl{\"u}cker coordinates $W^\star_{i_1,\dots,i_r}$ where $i_1+\dots+i_r=K+1$. These dual Pl{\"u}cker coordinates occour in Equation $(N=K+1)$:
\begin{align*}
\sum_{\substack{i_1+\dots+i_r=K+1\\ i_1<\dots<i_r}}h(i_1,\dots,i_r)\cdot W^\star_{i_1,\dots,i_r}=0
\end{align*}
Lemma~\ref{lem:szummas} says that for each pair $(W^\star_{i_1,\dots,i_r},W^\star_{j_1,\dots,j_r})$ of these dual Pl{\"u}cker coordinates above ($i_1+\dots+i_r=K+1=j_1+\dots+j_r$, where $\{i_1,\dots,i_r\}\neq\{j_1,\dots,j_r\}$, there exists a  a Pl{\"u}cker relation that has the form
\begin{equation*}
W^\star_{i_1,\dots,i_r}W^\star_{j_1,\dots,j_r}+\Sigma=0
\end{equation*}
where $\Sigma$ is the sum of some products $W_{k_1,\dots,k_r}H_{\ell_1,\dots,\ell_r}$, where $k_1+\dots+k_r<K+1<\ell_1+\dots+\ell_r$, and thus, using the assumption $i_1+\dots+i_r\le K\Rightarrow W^\star_{i_1,\dots,i_r}=0$, we get
\begin{equation*}
W^\star_{i_1,\dots,i_r}W^\star_{j_1,\dots,j_r}=0
\end{equation*}
for each pair $(W^\star_{i_1,\dots,i_r},W^\star_{j_1,\dots,j_r})$. These quadratic equations concludes that all $W^\star_{i_1,\dots,i_r}$ (where $i_1+\dots+i_r=K+1$) should be zero except one. And the linear Equation~($N=K+1$) says that this one cannot be exception either.

\paragraph{}
So we have proved that each dual Pl{\"u}cker coordinate of the subspace $W$ of co-dimension $r$ should be zero, that is a contradiction, since Pl{\"u}cker coordinates are homogeneous.

\paragraph{Opposite direction}
Suppose that $h(i_1,\dots,i_r)=0$ for a suitable $r$-tuple $i_1,\dots,i_r$ and consider the subspace $W$ of co-dimension $r$ co-ordinatized by the following dual Pl{\"u}cker coordinates. Let $W^\star_{i_1,\dots,i_r}=1$ for the $r$-tuple $i_1,\dots,i_r$ above, and let $W^\star_{j_1,\dots,j_r}=0$ for the other $r$-tuples of indices.
\begin{align*}
\sum_{j_1<\dots<j_r}W^\star_{j_1,\dots,j_r}\cdot H_{j_1,\dots,j_r}(t)=W^\star_{i_1,\dots,i_r}\cdot H_{i_1,\dots,i_r}(t)=1\cdot0\cdot t^{i_1+\dots+i_r-\binom{r}{2}}=0
\end{align*}
Thus, $W$ meets each $H(t)$ non-trivially.
\end{proof}

\begin{corollary}\label{cor:ind}
Suppose that $|\field|>r\cdot s$. The collection $\{H(t)\;|\;t\in\field\}$ of subspaces coordinatized by $H_{i_1,\dots,i_r}(t)=h(i_1,\dots,i_r)\cdot t^{i_1+\dots+i_r-\binom{r}{2}}$ is an $r$-uniform weak $(s,r\cdot s)$ subspace design \emph{if and only if} there does \emph{not} exists a subspace $W$ of rank $s$ (co-dimension $r$) meeting each $H(t)$ non-trivially, that is, if and only if $h(i_1,\dots,i_r)\neq0$ for each $r$-tuple $i_1,\dots,i_r$.
\end{corollary}
\begin{proof}
If $\{H(t)\;|\;t\in\field\}$ is \emph{not} an $r$-uniform weak $(s,r\cdot s)$ subspace design then $\exists W<\field^{r+s}$ ($\rank W=s$) such that $W$ meets \emph{more than} $r\cdot s$ elements $H(t)$ non-trivially. In this case the polynomial $\sum_{j_1<\dots<j_r}W^\star_{j_1,\dots,j_r}\cdot H_{j_1,\dots,j_r}(t)$ has \emph{more than} $r\cdot s$ roots, but its degree is $r\cdot s$, thus this must be the zero polynomial, and thus, $w$ meets \emph{all the elements} $H(t)$ non-trivially.

If $\{H(t)\;|\;t\in\field\}$ is an $r$-uniform weak $(s,r\cdot s)$ subspace design then $\forall W<\field^{r+s}$ of $\rank W=s$ meets at most $r\cdot s$ elements $H(t)$, and since $|\field|>r\cdot s$,  $\exists H(t)$ that meets $W$ only in the zero vector.
\end{proof}

Thus, if $H_{i_1,\dots,i_r}(t)=h(i_1,\dots,i_r)\cdot t^{i_1+\dots+i_r-\binom{r}{2}}$ and if we know that $\{H(t)\;|\;t\in\field\}$ is a weak $(s,A)$ subspace design with parameter $r\cdot s<A<|\field|$, then this corollary above prove that $\{H(t)\;|\;t\in\field\}$ is an $r$-uniform weak $(s,r\cdot s)$ subspace design, moreover, $h(i_1,\dots,i_r)\neq0$ for each $r$-tuple $i_1,\dots,i_r$. This will be used to show that known constructions has better (smaller) parameter $A$ than had been proved.

\subsection{Dual constructions}

Since we also will investigate constructions of {Guruswami and Kopparty}~\cite{VGSK}, we now give the connection between the techniques used in~\cite{VGSK} and the technique shown above.

Consider the collection $\{H(t)\;|\;t\in\field\}$ of subspaces co-ordinatized by the Pl{\"u}cker coordinate vectors $\mathbf{H}(t)=a^{[0]}(t)\wedge a^{[1]}(t)\wedge\dots\wedge a^{[k]}(t)$. Then the orthogonal complementary subspace of $H(t)$ is the intersection of hyperplanes: $H(t)^{\bot}=\{a^{[0]}(t)\}^{\bot}\cap\{a^{[1]}(t)\}^{\bot}\cap\dots\cap\{a^{[k]}(t)\}^{\bot}$.

The co-vector $b=[b_0,\dots,b_d]\in(\field^{d+1})^*$ is perpendicular to the coordinate vector $a^{[n]}(t)\in\field^{d+1}$ if and only if $\sum_{j=i}^{d}b_j\cdot h(i,n)\cdot t^{i-n}=0$. This motivates the following notations.
\begin{not*}
For the coordinate vector $z=[z_0,z_1,\dots,z_d]\in\field^{d+1}$ let $P_z(X)=\sum_{j=0}^{d}z_jX^j\in\field[X]\vphantom{\Big|}$ denote a univariate polynomial of degree at most $d$, and let $P_z^{[n]}(X)=\sum_{j=i}^{d}z_j\cdot h(n,i)\cdot X^{j-n}\in\field[X]$. Note that $P_z^{[0]}(X)=P_z(X)$.
\end{not*}
\begin{remark}
Note that $P_z^{[n]}(X)$ is a rational function, moreover, it is the quotient of a polynomial of degree $d$ with $X^n$. The function $P_z^{[n]}(X)$ is a polynomial (of degree $d-n$) if and only if $\forall i<n:h(i,n)=0$.
\end{remark}
The co-vector $b=[b_0,\dots,b_d]\in(\field^{d+1})^*$ is perpendicular to the homogeneous coordinate vector $a^{[n]}(t)\in\field^{d+1}$ if and only if $P_b^{[n]}(t)=0$, thus the annihillator subspace $H(t)^{\bot}=a^{[0]}(t)^{\bot}\cap a^{[1]}(t)^{\bot}\cap\dots\cap a^{[k]}(t)^{\bot}$ equals to the subspace $\left\{b\in(\field^{d+1})^*\;\Big|\;P_b^{[n]}(t)=0:\forall n=0,1,\dots,k\right\}\le(\field^{d+1})^*$.

Let the arbitrary linear subspace $W\le\field^{d+1}$ of rank $s$ be fixed and consider the set of polynomials $\left\{P_b(X)\;\Big|\;b\in W^{\bot}\right\}\subset\field[X]$. This subset is a linear subspace of $\field[X]$ and it is isomorphic to $W^{\bot}$ via the linear map $b\mapsto P_b(X)$. The constructions of {Guruswami and Kopparty}~\cite{VGSK} are based on this isomorphism.

Let $\{b(1),\dots,b(r)\}\subset(\field^{d+1})^*$ be a basis of $W^{\bot}$ and, using Gau{ss}ian elimination, without loss of generality we can suppose that the last $j-1$ coordinates of $b(j)$ are zero, that is, $\deg P_{b(j)}(X)\le d-j+1$.

The following matrix
\begin{equation*}
M(X)=\begin{bmatrix}
P_{b(1)}(X)&P_{b(2)}(X)&\dots&P_{b(r)}(X)\\
P_{b(1)}^{[1]}(X)&P_{b(2)}^{[1]}(X)&\dots&P_{b(r)}^{[1]}(X)\\
\vdots&\vdots&\ddots&\vdots\\
P_{b(1)}^{[k]}(X)&P_{b(2)}^{[k]}(X)&\dots&P_{b(r)}^{[k]}(X)\\
\end{bmatrix}
\end{equation*}
is quadratic since $r=k+1$. The linear map $b\mapsto P_b(t)$ maps the subspace $H(t)^{\bot}\cap W^{\bot}$ to the kernel of $M(t)$. If $h(i,n)=0:\forall i<n$ then this matrix is a polynomial matrix (each element is a univariate polynomial), thus, its determinant is also a polynomial.

\begin{lemma}\label{lem:matrix}
Let $M(X)\in\field[X]^{m\times m}$ be an $m\times m$ matrix of polynomials. For each element $t\in\field$, if $\det M(t)=0$ then $t$ is the root of the polynomial $\det M(x)$ of multiplicity at least $R=\rank\ker(M(t))$.
\end{lemma}
\begin{proof}
Let $t\in\field$ be an arbitrary element of the field and consider the matrix $M(t)\in\field^{m\times m}$ and its kernel $\ker(M(t))\le\field^m$ of rank $R=\rank\ker(M(t))$. Let $a^{(1)},\dots,a^{(R)},a^{(R+1)},\dots,a^{(m)}$ be such a basis of $\field^m$ that $a^{(1)},\dots,a^{(R)}$ is the basis of $\ker(M(t))$ and $\det A=1$ where $A=[a^{(1)},\dots,a^{(m)}]$.
\begin{equation*}
M(t)A=\begin{bmatrix}
0&\dots&0&*&\dots&*\\
0&\dots&0&*&\dots&*\\
\vdots&&\vdots&\vdots&&\vdots\\
0&\dots&0&*&\dots&*\\
\end{bmatrix}
\end{equation*}
Thus, the elements of the first $R$ coloumns of the matrix $M(X)A$ are polynomials vanishing on $X=t$, thus the linear polynomial $(X-t)$ divides them. So $\det(M(X))=\det(M(X)A)=(X-t)^R\cdot f(X)$ where $f(X)\in\field[X]$.
\end{proof}

\begin{lemma}\label{lem:dual}
Suppose that $h(i_1,\dots,i_r)\neq0$ for all $\{i_1,\dots,i_r\}$ and suppose that $h(i,n)=0:\forall i<n$. The collection $\{H(t)^\bot\;|\;t\in\field\}$ of subspaces perpendicular to the subspaces co-ordinatized by the Pl{\"u}cker coordinates $H_{i_1,\dots,i_r}(t)=h(i_1,\dots,i_r)\cdot t^{i_1+\dots+i_r-\binom{r}{2}}$ is an $s$-uniform strong $(r,s\cdot r)$ subspace desing.

Moreover, in this case, the collection $\{H(t)\;|\;t\in\field\}$ of subspaces co-ordinatized by the Pl{\"u}cker coordinates $H_{i_1,\dots,i_r}(t)=h(i_1,\dots,i_r)\cdot t^{i_1+\dots+i_r-\binom{r}{2}}$ is an $r$-uniform strong $(s,r\cdot s)$ subspace desing.
\end{lemma}
\begin{proof}
Using Lemma~\ref{lem:matrix} above we can see that
\begin{equation*}
\sum_{t\in\field}\rank(H(t)^{\bot}\cap W^{\bot})\le\sum_{t\in\field}\mult(\det M(X),t)\le\deg\det M(X)
\end{equation*}
if $\det M(X)\neq0$ (if it is not the zero polynomial).

Since $h(i_1,\dots,i_r)\neq0$ for all $\{i_1,\dots,i_r\}$, thus, Corollary~\ref{cor:ind} says that $\{H(t)\;|\;t\in\field\}$ is an $r$-uniform weak $(s, r\cdot s)$ subspace design. Theorem~\ref{thm:weakdual} says that in this case $\{H(t)^\bot\;|\;t\in\field\}$ is an $s$-uniform weak $(r,s\cdot r)$ subspace design. So, for each subspace $W$ of rank $s$, the subspace $W^\bot$ of rank $r$ does \emph{not} block all the subspaces $H(t)^\bot$, thus, the polynomial $\det M(X)$ cannot be zero for all substitution $X=t$.

Remember that $\deg P_{b(j)}\le d-j+1$, and we know that for arbitrary polynomial $P(X)$ the degree $\deg P^{[i]}(X)=\deg P(X)-i$. Since $\det M(X)=\sum_{\sigma\in S_s}(-1)^{I(\sigma)}\prod_{j=1}^sP_{b(\sigma j)}^{[j-1]}(X)$, $\deg\det M(X)\le\sum_{j=1}^s(d-\sigma j+1-(j-1))=s\cdot d-\sum_{j=1}^s(\sigma j)-\sum_{j=1}^sj=s\cdot d-2\binom{s}{2}=s\cdot(d-s+1)$.

Thus, 
\begin{equation*}
\sum_{t\in\field}\rank(H(t)^{\bot}\cap W^{\bot})\le\deg\det M(X)\le
s\cdot(d-s+1)
\end{equation*}

So the collection $\{H(t)^{\bot}:t\in\field\}$ of subspaces in $(\field^{s+r})^*$ is an $s$-uniform $(r,s\cdot r)$ strong subspace design, where $r=d-s+1$.

The last statement comes directly from Theorem~\ref{thm:strongdual}.
\end{proof}

Finally, we consider the particular constructions, at first the most basic construction, the tangents of the moment curve.

\subsection{Tangents of the moment curve}\label{subs:tang}

In this subsection we suppose that the the characteristic of the field $\field$ is bigger than $r+s$, since the derivates could vanish otherwise, making errors in the proofs.

Let $h(i,n)=\frac{i!}{(i-n)!}$ if $i\ge n$, and let $h(i,n)=0$ if $i<n$. Then $a^{[n]}(t)$ is the $n$-th derivate of $a(t)$ as a $(d+1)$-tuple of polynomials of variable $t$. The subspace $H(t)$ co-ordinatized by the Pl{\"u}cker coordinate vector $\mathbf{H}(t)=a^{[0]}(t)\wedge a^{[1]}(t)\wedge\dots\wedge a^{[k]}(t)$ is the `tangent subspace of rank $r$' of the moment curve.

The dual construction $\{H(t)^\bot\;|\;t\in\field\}$ is exactly the basic construction of {Guruswami and Kopparty}~\cite[Subsection~5.1]{VGSK} based on multiplicity codes. Thus, Theorem~\ref{thm:GKimp} says that $\{H(t)^\bot\;|\;t\in\field\}$ is an $s$-uniform strong $(r,r\cdot s)$ subspace design.

Thus, according to Theorem~\ref{thm:strongdual}, $\{H(t)\;|\;t\in\field\}$ is an $r$-uniform strong $(s,r\cdot s)$ subspace design. And thus, it is also an $r$-uniform weak $(s,r\cdot s)$ subspace design. So, we have seen that $h(i_1,\dots,i_r)\neq0$ for each $r$-tuple $i_1,\dots,i_r$ if the characteristic of the field $\field$ is big enough. The following constructions are made to eliminate the problem of small characteristics, so, from now on, the characteristic of the field $\field$ is again arbitrary.

\subsection{Diverted tangents of the moment curve}\label{subs:div}

In our previous article~\cite{higpig}, we solve the problem of small characteristics by `diverting' the tangent lines of the moment curve. The `almost generalization' of this idea is the following. (Almost, because the case $r=2$ is not exactly the same that the `diverted tangent lines' in that article, but the technique is very similar.)

Let $\omega\in\field\setminus\{0\}$ be a suitable element and let $h(i,n)=(\omega^n)^{i-r}$ if $i\ge r$ or if $i=n<r$, and let $h(i,n)=0$ otherwise. So, the elements $h(i_1,\dots,i_r)$ are the $r\times r$ subdeterminants of the following $r\times d$ matrix.
\begin{equation*}
\begin{bmatrix}
1&0&0&\dots&0&0&1&1&\dots&1\\
0&\omega^{1-r}&0&\dots&0&0&1&\omega^1&\dots&\omega^{d-r}\\
0&0&\omega^{2(2-r)}&\dots&0&0&1&\omega^2&\dots&\omega^{2(d-r)}\vphantom{\Big|}\\
\vdots&\vdots&\vdots&\ddots&\vdots&\vdots&\vdots&\vdots&&\vdots\\
0&0&0&\dots&\omega^{(r-2)(-2)}&0&1&\omega^{r-2}&\dots&\omega^{(r-2)(d-r)}\\
0&0&0&\dots&0&\omega^{(r-1)(-1)}&1&\omega^{r-1}&\dots&\omega^{(r-1)(d-r)}\\
\end{bmatrix}
\end{equation*}
Consider the $r$-tuple of indices $0\le i_1<\dots<i_r\le d$ where $i_R\le r-1<r\le i_{R+1}$ and let $0\le j_1<\dots<j_{r-R}\le r-1$ denote the integers such that $\{i_1,\dots,i_R,j_1,\dots,j_{r-R}\}=\{0,1,\dots,r-1\}$. Using these notations one can easily see that
\begin{align*}
h(i_1,\dots,i_r)&=\pm\omega^{i_1(i_1-r)+\dots+i_R(i_R-r)}\cdot\det\Omega&\mbox{ where }&
&\Omega_{k\ell}&=\omega^{j_{\ell}\cdot(-r+i_{R+k})}.
\end{align*}
In particular, $h(0,1,\dots,r-1)=\omega^{0+(1-r)+(4-2r)+\dots+(1-r)}\neq0$. If $i_1\ge r$ then $h(i_1,\dots,i_r)=V(\omega^{i_1-r},\dots,\omega^{i_r-r})$ and if $i_1=r-1$ then $h(i_1,\dots,i_r)=\pm\omega^{1-r}\cdot V(\omega^{i_2-r},\dots,\omega^{i_r-r})$, where $V(a_1,\dots,a_N)$ denotes the Vandermonde determinant, which is nonzero if and only if the elements $a_1,\dots,a_N$ are distinct. So $\omega$ has to be an element of order more than $d-r=s-1$.

There is a little problem with the numbers $h(i_1,\dots,i_r)$ if some of the indices are less than $r-1$ and some of them are bigger. The determinant of $\Omega$ is a \emph{generalized Vandermonde determinant}.

\begin{definition}
Let $0\le j_1<j_2<\dots<j_N$ be a strictly increasing series of non-negative integers and let $a_1,\dots,a_N$ be elements of the field $\field$. The \emph{generalized Vandermonde} matrix and determinant is defined as the matrix $V$ of entries $V_{k\ell}=a_k^{j_\ell}$ and its determinant, respectively. If $j_i=i-1$ then they are the well known Vandermonde matrix and determinant.
\end{definition}

We have a very special case of \emph{generalized Vandermonde} matrices here, where $a_k=\omega^{-r+i_{R+k}}$.

\begin{remark}
The generalized Vandermonde determinant is totally positive over $\mathbb{R}$ if $a_1<\dots<a_r$ are distinct positive elements of $\mathbb{R}$. So, if $\mathbb{Q}\subset\field$ (that is, if $\karak\field=0$) then $\omega\in\mathbb{Q}$, $\omega>1$ is a suitable element. But, if $\karak\field=p\neq0$ then  a generalized Vandermonde determinant of distinct elements can be zero.
\end{remark}

Since the generalized Vandermonde determinant is an alternating multivariate polynomial of the variables $a_1,\dots,a_N$, it is the product of the Vandermonde determinant $V(a_1,\dots,a_N)$ and a symmetric polynomial of these variables.

Let $N$ denote $r-R$ and let $b_k=i_{R+k}-r$. Our generalized Vandermonde determinant
\begin{equation*}
\det\Omega=\sum_{\sigma\in S_N}(-1)^{I(\sigma)}\prod_{k=1}^N\omega^{(-r+i_{R+k})\cdot j_{\sigma k}}
=\sum_{\sigma\in S_N}(-1)^{I(\sigma)}\omega^{\sum_{k=1}^N b_k\cdot j_{\sigma k}}
\end{equation*}
which is a univariate polynomial of $\omega$ and the polynomial
$\det\Omega$ is divisible by the Vandermonde determinant
\begin{equation*}
V(\omega^{b_1},\dots,\omega^{b_N})=\prod_{i<j}(\omega^{b_j}-\omega^{b_i})
\end{equation*}
which is also a univariate polynomial of $\omega$. The following lemma gives degrees of these polynomials.
\begin{lemma}
Let $0\le j_1<j_2<\dots<j_N\le r-1$ an let $0\le b_1<b_2<\dots<b_N\le d-r$ be integers and for each permutation $\sigma\in S_N$ let $\Sigma(\sigma)$ denote the sum
\begin{equation*}
\Sigma(\sigma)=\sum_{k=1}^{N}b_k\cdot j_{\sigma k}
\end{equation*}
Using these notations, for each non-identical permutation $\sigma\neq\ident:\Sigma(\sigma)<\Sigma(\ident)\le N\cdot(d-r)\cdot(r-1)-\binom{N}{2}\cdot\frac{d}{3}$. Moreover, $\max\Sigma(\sigma)-\min\Sigma(\sigma)\le\frac{N(d-r)(r-1)}{2}$.
\end{lemma}
\begin{proof}
The strict inequality $\Sigma(\sigma)<\Sigma(\ident)$ comes by induction from the following fact. Suppose that $\sigma k>\sigma(k+1)$. (Such a $k$ exists if and only if $\sigma\neq\ident$.) Then let $\sigma'\ell=\sigma\ell$ for each $\ell\neq k,\ell\neq(k+1)$, and let $\sigma'k=\sigma(k+1)<\sigma k=\sigma'(k+1)$. Then $\Sigma(\sigma')-\Sigma(\sigma)=b_k\cdot j_{\sigma(k+1)}+b_{k+1}\cdot j_{\sigma k}-b_k\cdot j_{\sigma k}+b_{k+1}\cdot j_{\sigma(k+1)}=(b_{k+1}-b_k)\cdot(j_{\sigma k}-j_{\sigma(k+1)})>0\cdot0$.
By induction in the number of inversions we get the result. The same
induction shows that
$\min\Sigma(\sigma)=\Sigma(\opp)=\sum_{k=1}^{N}b_k\cdot j_{N+1-k}$
where $\opp\in S_N$ denotes the opposite permutation.

\paragraph{First bound}
$\Sigma(\ident)=\sum_{k=1}^N b_k\cdot
j_k\le\sum_{k=0}^{N-1}(d-r-k)(r-1-k)=\sum_{k=0}^{N-1}\big((d-r)(r-1)-k(d-1)+k^2\big)=
N(d-r)(r-1)-\binom{N}{2}(d-1)+\frac{(N-1)N(2N-1)}{6}=N(d-r)(r-1)+\binom{N}{2}(\frac{2N-1}{3}-d+1)=N(d-r)(r-1)-\binom{N}{2}\frac{3d-2(N+1)}{3}$.
Since $N+1=r-R+1\le d$, the first bound comes.

\paragraph{Second bound}
$2\big(\Sigma(\ident)-\Sigma(\opp)\big)=\sum_{k=1}^{N}j_k\cdot(b_k-b_{N+1-k})+\sum_{k=1}^N j_{N+1-k}\cdot(b_{N+1-k}-b_k)=\sum_{k=1}^{N}(j_k-j_{N+1-k})(b_k-b_{N+1-k})\le\sum_{k=1}^N(r-1)(d-r)=N(r-1)(d-r)$.
\end{proof}

\begin{corollary}\label{cor:deg}
Since $\Sigma(\sigma)<\Sigma(\ident)$ if $\sigma\neq\ident$, the leading term of the univariate polynomial $\det\Omega$ is $\omega^{\Sigma(\ident)}$, thus, this polynomial is \emph{not the zero polynomial} and its degree is $\max\Sigma(\sigma)\le N\cdot(d-r)\cdot(r-1)-\binom{N}{2}\cdot\frac{d}{3}$, moreover, $\det\Omega$ is the product of $\omega^{\Sigma(\opp)}$ and a polynomial of degree at most $\max\Sigma(\sigma)-\min\Sigma(\sigma)\le\frac{N(d-r)(r-1)}{2}\le\binom{r}{2}(d-r)$.

Since $N\le r$, the polynomial $\det\Omega$ has at most $\binom{r}{2}(d-r)$ nonzero roots. Moreover, since the coefficients of $\det\Omega$ are the sums of $\pm1$, its each root is in the extension field of the prime field $\field_p$ of degree at most $\binom{r}{2}(d-r)$, so their order is at most $p^{\binom{r}{2}(d-r)}-1$.\qed
\end{corollary}

Thus, the requirements $h(i_1,\dots,i_r)\neq0$ are polynomial conditions for the element $\omega$. If $\omega$ is \emph{not} the soultion \emph{any} of the equations $h(i_1,\dots,i_r)=0$, then $h(i_1,\dots,i_r)\neq0$ for each $r$-tuple $i_1,\dots,i_r$.

\begin{theorem}
Let $r$ and $s$ be given and suppose that $\karak\field=p\neq0$. If $\field$ has more than $\binom{r+s}{r}\binom{r}{2}(s-1)$ elements, or if the field has $q=p^h$ elements and $h>\binom{r}{2}(s-1)$, then there exists an $r$-uniform strong $(s,r\cdot s)$ subspace design, and an $s$-uniform strong $(r,s\cdot r)$ subspace design in the vector space $\field^{r+s}$.
\end{theorem}
\begin{proof}
Let $d=r+s-1$ and suppose that $\field$ has more than $p^{\binom{r}{2}(d-r)}$ elements. If $\field=\field_q$, where $q=p^h$, $h>\binom{r}{2}(d-r)$, then the order of the primitive element $\omega\in\field_q$ is $p^h-1>p^{\binom{r}{2}(d-r)}-1$. If $\field$ is infinite then it has element $\omega$ of order more than $p^{\binom{r}{2}(d-r)}-1$. Corollary~\ref{cor:deg} above yields that if the order of $\omega$ is bigger than $p^{\binom{r}{2}(d-r)}-1$, then $\omega$ cannot be the solution of \emph{any} of the equations $h(i_1,\dots,i_r)=0$.

Suppose that $\field$ has more than $\binom{d+1}{r}\binom{r}{2}(d-r)$ elements. Then the number of distinct roots of all the polynomials $\det\Omega$ (for all $r$-tuples) is smaller than $|\field|$ so $\exists\omega\in\field$ such that $\omega$ is \emph{not} the solution of \emph{any} of the equations $h(i_1,\dots,i_r)=0$.

If we choose the element $\omega$ properly, Lemma~\ref{lem:dual} says that the set of diverted tangents of the moment curve is an $r$-uniform strong $(s,r\cdot s)$ subspace design, and the orthogonal complementary subspaces of the diverted tangents constitutes an $s$-uniform strong $(r,s\cdot r)$ subspace design.
\end{proof}

If the characteristic $p$ of the field $\field_q$ is smaller than $r+s$ and $q<\binom{r+s}{r}\binom{r}{2}(s-1)$ and $q=p^h$ and $h\le\binom{r}{2}(s-1)$ then the theorem above does not work. In this case we have to use another construction yielding strong subspace designs of parameter $A$ bigger than $r\cdot s$, but we will see that this parameter $A$ is $r\cdot s$ if we consider the subspace design as a \emph{weak} subspace design.

\subsection{Secants of the moment curve}\label{subs:sec}

Consider the diverted tangents of the moment curve. Since $V(\omega^{i_1},\dots,\omega^{i_r})\cdot t^{i_1+\dots+i_r-\binom{r}{2}}$ is not meaningless if some index $i_k<r$, we can extend the definition of $H_{i_1,\dots,i_r}(t)=V(\omega^{i_1},\dots,\omega^{i_r})\cdot t^{i_1+\dots+i_r-\binom{r}{2}}$ to the all $r$-tuples of indices. In this case $h(i,n)=(\omega^n)^i$ for all $i$ and $n$.

Since for $t\neq0$ $H(t)$ is the subspace containing the points of the moment curve $a(t),a(\omega t),a(\omega^2 t),\dots,a(\omega^{r-1}t)$ (remember that the vector $\frac{1}{t^{n}}a(\omega^{n} t)$ and the vector $a(\omega^{n} t)$ co-ordinatized the same projective point if $t\neq0$) these subspaces will be called the $\omega$-secants of the moment curve.

Since $i_r\le d=r+s-1$, $V(\omega^{i_1},\dots,\omega^{i_r})\neq0$ if (and only if) the order of $\omega$ is at least $r+s$. Suppose that $|\field|>r+s$ and let $\omega\in\field$ such a suitable element. Corollary~\ref{cor:ind} of Lemma~\ref{lem:ind} says that the set $\{H(t)\;|\;t\in\field\}$ of the $\omega$-secants of the moment curve is an $r$-uniform weak $(s,r\cdot s)$ subspace design.

Consider the orthogonal complementary subspaces of the $\omega$-secants. It is exactly the basic construction of {Guruswami and Kopparty}~\cite[Subsection~4.1]{VGSK} based on Reed--Solomon codes and Theorem~\ref{thm:GKimp} says that $\{H(t)^\bot\;|\;t\in\field\}$ is an $s$-uniform strong $(r,s\cdot r+\binom{r}{2})$ subspace design. Thus, according to Theorem~\ref{thm:strongdual}, the set $\{H(t)\;|\;t\in\field\}$ of the $\omega$-secants of the moment curve is an $r$-uniform strong $(s,r\cdot s+\binom{r}{2})$ subspace design.

According to Theorem~\ref{thm:weakdual}, Guruswami--Kopparty construction $\{H(t)^\bot\;|\;t\in\field\}$ based on Reed--Solomon codes is an $s$-uniform weak $(r,r\cdot s)$ subspace design.


\section*{Summary}

Let the natural numbers $r\ge2$ and $s\ge2$ be given and let $\field$ be an arbitrary field of more than $r+s$ elements. Then the constructions above prove that there exist $r$-uniform strong $(s,r\cdot s+\min\{\binom{r}{2},\binom{s}{2}\})$ subspace designs in $\field^{r+s}$. Moreover, there exist $r$-uniform strong $(s,r\cdot s)$ subspace designs in $\field^{r+s}$
\begin{itemize}
\item if $\karak\field=0$, or
\item if $\karak\field=p>r+s$, or
\item if $|\field|>\binom{r+s}{r}\binom{r}{2}(s-1)$ or
\item if $|\field|>p^{\binom{r}{2}(s-1)}$ where $p=\karak\field$.
\end{itemize}
The constructions also prove that there exist $r$-uniform weak $(s,r\cdot s)$ subspace designs in $\field^{r+s}$, thus, for arbitrary $k\ge1$ and $d\ge3$ there exist $(k\!+\!1)\cdot(d\!-\!k)+1$ projective $k$-subspaces of $\PG(d,\field)$ in higgledy-piggledy arrangement.

\subsection*{Some open question}

The construction of $(k\!+\!1)\cdot(d\!-\!k)+1$ projective $k$-subspaces of $\PG(d,\field)$ in higgledy-piggledy arrangement is the smallest one over algebraically closed field $\field$. Over other fields we have a much smaller lower bound, but we do not know whether there are smaller sets of higgledy-piggledy $k$-subspaces or not. We do not know the tight lower bound over non-closed fields.

We prove that the diverted tangents of the moment curve is a good construction if the field has more than $\binom{r+s}{r}\binom{r}{2}(s-1)$ elements or more than $p^{\binom{r}{2}(s-1)}$ elements where $p=\karak\field$, but we do not know whether this construction works well also over some smaller fields. We conjecture that it does.


\end{document}